\def\N{\mathbb N}
\def\Z{\mathbb Z}
\def\R{\mathbb R}
\def\Q{\mathbb Q}
\def\C{\mathbb C}
\def\A{\mathcal A}
\def\B{\mathcal B}
\def\Oo{\mathcal{O}}
\def\D{\mathcal{D}}
\def\K{\mathbb{K}}

\def\frp{\operatorname{frp}}
\def\Per{\operatorname{Per}}

\def\inp{\operatorname{inp}}
\def\frp{\operatorname{frp}}
\def\p{\mathfrak{p}}

\def\pfz{\begin{proof}}
\def\pfk{\end{proof}}
\documentclass[12pt]{amsart}
\usepackage[lmargin=2.5cm, rmargin=2.5cm,bottom=2.5cm,top=2.5cm]{geometry}
\usepackage{amssymb,amsmath,mathdots}
\usepackage{graphicx}

\usepackage{epstopdf}

\usepackage{subfigure}
\usepackage{comment}
\usepackage{enumerate}

\newtheorem{theorem}{Theorem}[section]
\newtheorem{proposition}[theorem]{Proposition}

\newtheorem{lemma}[theorem]{Lemma}
\newtheorem{defn}[theorem]{Definition}

\begin{document}
\title{Periodic representations in Salem bases}
\author[AMS]{Tom\'a\v s V\'avra}
\address{Department of Algebra\\
Charles University\\
Sokolovsk\'a 83\\
186 75 Praha 8\\
Czech Republic}
\email{vavrato@gmail.com}

\begin{abstract}
We prove that all algebraic bases $\beta$ allow an eventually periodic representations of the elements of $\Q(\beta)$ with a finite alphabet of digits $\A$. Moreover, the classification of bases allowing that those representations have the so-called weak greedy property is given.

The decision problem whether a given pair $(\beta,\A)$ allows eventually periodic representations proves to be rather hard, for it is equivalent to a topological property of the attractor of an iterated function system.
\end{abstract}
\maketitle

\section{Introduction}
The authors of \cite{BMPV} studied the following problem: for which algebraic bases $\beta$, $|\beta|>1$, there is a finite alphabet of digits $\A,$ such that each $x\in\Q(\beta)$ can be expressed as
\[x=\sum_{i=-L}^{+\infty}a_i\beta^{-i},\ a_i\in\A,\quad(a_i)_{i=-L}^{+\infty}\text{ eventually periodic}.\]\
This problem is a generalization of a well known property of Pisot bases. Indeed, K. Schmidt in \cite{Schmidt} proved that the greedy $\beta$-expansions of $\Q(\beta)\cap\R^+$ are eventually periodic. It is worth mentioning that Schmidt conjectured that this holds also for the $\beta$-expansions in Salem bases, however, this has not been proved for a single instance of a Salem base so far.

The problem was partially solved in \cite{BMPV} for a certain subclass of algebraic numbers. In the subsequent paper \cite{KV}, it was proved that all algebraic bases without conjugates on the unit circle allow eventually periodic representations with some alphabet. The proofs relied heavily on the  existence of parallel addition algorithms on $(\beta,\A)$-representations. The drawback of this method is that the parallel algorithms are not available when the base has a conjugate on the unic circle, hence it cannot be used to, for instance, Salem bases. Nevertheless, by generalizing the Fermat's little theorem, the authors in \cite{KV} were able to prove that in any algebraic base and a suitable digit alphabet, the number $\tfrac1n$ have an eventually periodic representation for any $n\in\N.$

In this paper we will solve the problem completely, as we show that every algebraic base $\beta$ allows eventually periodic representations of $\Q(\beta)$ with some finite digit alphabet. We use a rather number theoretical approach. In particular, we consider the embeddings of $\Q(\beta)$ corresponding to all places $\p$ such that $|\beta|_\p>1.$ A similar approach for different problems connected to the theory of number systems was used for example in \cite{ATZ} and \cite{SteThu}.

A related problem is to decide whether a given pair $(\beta,\A)$ allows eventually periodic representations of $\Q(\beta)$. We show that this is related to attractors of certain iterated function systems, as well as to a geometric property of the so-called spectrum of $\beta$ with the alphabet $\A$.
\section{Definitions and main results}
Let the base $\beta\in\C$ be such that $|\beta|>1$ and let $\A\subset\C$ be a finite digit alphabet. By a $(\beta,\A)$-representation of $x\in\C$ we mean the expression of the form 
$x = \sum_{i=-L}^{+\infty} a_i\beta^{-i}.$ A particular representation is said to be eventually periodic if the sequence $(a_i)_{i=-L}^{+\infty}$ is eventually periodic.
The set of numbers admitting an eventually periodic $(\beta,\A)$-representation is denoted $\mathrm{Per}_\A(\beta),$ i.e.
\[
\mathrm{Per}_\A(\beta) = \{x\in\C : x = \sum_{i=-L}^{+\infty} a_i\beta^{-i}, (a_i)_{i=-L}^{+\infty}\text{ eventually periodic}\}.
\]

There are several ways of constructing $(\beta,\A)$-representations, most notable of them being the well-known greedy $\beta$-expansions for real bases $\beta>1$ introduced in \cite{ren}. A rather general concept was given by Thurston \cite{Thurston}. For $V\subset\C$ bounded and $\A\subset\C$ finite, let the condition $\beta V\subseteq \bigcup_{a\in\A} (V+a)$ be satisfied. A $(\beta,\A)$-representation of an element of $V$ can then be then constructed as follows. Define a transformation $T:V\rightarrow V$ by
\begin{equation}\label{eq:thurston}T(x) = \beta x-D(x)\quad\text{ with } D(x)\in\A.\end{equation}
Then $x = D(x)\beta^{-1} + {T(x)}\beta^{-1}$, and by iterating this procedure and denoting $a_i = D(T^{i-1}(x))$ we obtain $x = \sum_{i=1}^{+\infty}a_i\beta^{-i}$. 
Moreover, it can be easily seen that if the sequence $(T^n(x))_{n\geq0}$ takes only finitely many values, then the corresponding $(\beta,\A)$-representation is eventually periodic. We will later prove the following result by generalizing the Thurston's construction.
\begin{theorem}\label{thm:main1}
Let $\beta\in\C$, $|\beta|>1$ be an algebraic number. Then there exists $\A\subset\Z$ finite such that $\mathrm{Per}_\A(\beta)=\Q(\beta)$.
\end{theorem}

Another property of representations studied in~\cite{BMPV} was weak-greediness. In that context, special classes of algebraic numbers turned to be important. An algebraic integer $\beta>1$ is called a Pisot number if its Galois conjugates $\beta'$ satisfy $|\beta'|<1$. If an algebraic integer $\beta>1$ satisfies $|\beta'|\leq1$ with at least one conjugate lying on the unit circle, then $\beta$ is called a Salem number. We call complex Pisot or complex Salem numbers the corresponding complex analogies where $\beta>1$ is replaced by $|\beta|>1$ and where the condition on the Galois conjugates does not hold for the complex conjugation.
\begin{defn}
We say that $(\beta,\A)$ admits weak-greedy eventually periodic representations of $\Q(\beta)$ if there is $c>0$ such that every $x\in\Q(\beta), |x|<c$ allows an eventually periodic $(\beta,\A)$-representation of the form $x=\sum_{i=1}^{+\infty}a_i\beta^{-i}$.
\end{defn}
Weak-greediness means, roughly speaking, existence of a representation whose highest power is proportional to the modulus of the represented number.
It was shown in \cite{BMPV} that if $(\beta,\A)$ admits weak-greedy eventually periodic representations of $\Q(\beta)$, then $\beta$ is either a (complex) Pisot or a (complex) Salem number, or all the conjugates $\beta'$ of $\beta$ outside the unit circle satisfy $|\beta'|=|\beta|$. Furthermore, it was shown that all the (complex) Pisot bases allow weak-greedy eventually periodic representations. We give the full classification in the following theorem.
\begin{theorem}\label{thm:wg}
Given a base $\beta\in\C$, there exists $\A\subset\Q(\beta)$ such that $(\beta,\A)$ admits weak-greedy eventually periodic $(\beta,\A)$-representations of $\Q(\beta)$ if and only if $\beta$ is an algebraic integer without Galois conjugates outside of the unit circle other than itself and its complex conjugate.
\end{theorem}

It is natural to ask whether a given pair $(\beta,\A)$ admits eventually periodic representations of $\Q(\beta)$. We give several equivalent conditions in Theorem~\ref{thm:main2}. Before stating it, we will need to introduce some number theoretical notation.
We will follow the notation from \cite{ATZ} and \cite{SteThu}, although there is a difference in the definition of $S_\beta$ (because we will not be working only with expansive numbers).

Let $\beta$ be an algebraic number, and denote $K = \Q(\beta)$ with the ring of integers $\Oo_K$. Let $\beta\Oo_K = \frac{\mathfrak a}{\mathfrak b}$ with $\mathfrak a, \mathfrak b$ being coprime ideals in $\Oo_K$. Define a set of places of $K$
\[S_\beta = \{\p : \p\ |\ \infty \text{ and } |\beta|_\p\geq 1 \}\cup\{\p : \p\ |\ \mathfrak b\}.\]
Furthermore, let $\K_\beta = \prod_{\p\in S_\beta} K_\p$ where
$K_\p$ denotes the completion of $K$ with respect to the $\p$-adic norm. The space $\K_\beta$ is endowed with the norm
$|x|_\beta = \max\{|x|_\p : \p\in S_\beta\}$ and with the respective topology.

When we speak about elements of $\Q(\beta)$ in $\K_\beta$, we mean their images through the diagonal embedding 
\[\Phi_\beta: \Q(\beta)\rightarrow \K_\beta,\quad x\mapsto \prod_{\p\in S_\beta}x.\]
When no confusion is expected, the symbol $\Phi_\beta$ will be ommited. 
Of course, a generic point of $\K_\beta$ does not correspond
to any element of $\Q(\beta)$. Nevertheless, an approximation of $\K_\beta$ by $\Q(\beta)$ is possible. The following proposition is a direct application of the well known weak approximation theorem (see for instance Theorem 3.4 of \cite{Neu}).
\begin{proposition}\label{prop:weak}
$\Q(\beta)$ is dense in $\K_\beta$. In other words, for any $\varepsilon>0$ and any $x\in\K_\beta$  there is $z\in\Q(\beta)$ such that $|x-\Phi_\beta(z)|_\beta <\varepsilon.$
\end{proposition}

By the Hutchinson's theorem on iterated function systems (see~\cite{hutchinson}), there exists a unique non-empty compact set $K(\beta,\A)\subset\K_\beta$ satisfying 
\[K(\beta,\A) = \beta^{-1}\left(\bigcup_{a\in\A}K(\beta,\A)+a\right).\]
The iterated funtion system consists of the contracting maps $x\mapsto\beta^{-1}(x + a)$ on the complete metric space $\K_\beta$.
The set $K(\beta,\A)$ is usually called the attractor of the iterated function system
We see that the attractor can be alternatively described as
\[
K(\beta,\A) = \left\{\sum_{i=1}^{+\infty}\Phi_\beta(a_i\beta^{-i}) : a_i\in\A\right\}.
\]
Note that the compactness of $K(\beta,\A)$ can be then alternatively proved, as in~\cite{SteThu}, by $K(\beta,\A)$ being a continuous image of the compact space of infinite words 
$\A^\N= \{a_1a_2a_3\dots : a_i\in\A\}.$

Another notion we need to introduce is the spectrum of $\beta$ with the alphabet $\A$ as introduced by Erd\H{o}s, J\'oo, and Komornik in~\cite{ejk90}. Note that the original definition was given for $1<\beta<2$, $\A=\{0,1\}$ only.

Let $\beta\in\C$, $|\beta|>1$ and let $\A\subset\C$ be finite. We set
\[X^\A(\beta) = \left\{\sum_{i=0}^{n}a_i\beta^i : n\in\N, a_i\in\A\right\}.\] 
Many authors contributed to the study of $X^\A(\beta)$, namely to the following two properties.
We say that a set $X\subset \mathbb K$ is:
\begin{enumerate}
  \item uniformly discrete, if $0$ is not an accumulation point of $X-X$;
  \item relatively dense, if there exists $R>0$ such that for every $z\in\K$ we have $B_R(z)\cap X\neq \emptyset.$
\end{enumerate}
If both conditions are satisfied, then $X$ is said to be a Delone set. 
The question when is the spectrum of a real $\beta>1$ with an integer alphabet $\{0,1,\dots,m\}$ a Delone set in $\R^+$ was completely solved recently in~\cite{feng}.
Finer results on the structure of gaps of $X^\A(\beta)$, their lengths or frequencies, were given for instance in~\cite{bugeaud,hare,mapape}.
Spectra of complex bases $\beta$ with integer alphabet were considered in~\cite{zaimi,hepe}.

Our result stated as Theorem~\ref{thm:main2} puts into relation the relative density of the spectrum $X^\A(\beta)$, the attractor $K(\beta,\A)$, and the possibility of periodic $(\beta,\A)$-representations of $\Q(\beta)$ and $\Z[\beta]$.

\begin{theorem}\label{thm:main2}
Let $\beta$ be an algebraic number without conjugates on the unit circle, and let $\A\subset\Q(\beta)$ be finite. The following statements are equivalent.
\begin{enumerate}
  \item $\Q(\beta)\subseteq\Per_\A(\beta)$;
  \item $\Z[\beta]\subseteq\Per_\A(\beta)$;
  \item The spectrum $X^\A(\beta)$ is relatively dense in $\mathbb K_\beta$.
  \item $0\in\operatorname{int}(K(\beta,\A))$ in $\K_\beta$.
\end{enumerate}
\end{theorem}
The strength of Theorem~\ref{thm:main2} is that is connects objects that were already studied in the literature.
A special case of the equivalence of (3) and (4) in case $\K_\beta = \C$ was stated in~\cite{HMV}. 
Tiling properties of the attractors $K(\beta,\A)$ with $\beta$ expanding and with specific digit alphabets were studied in~\cite{SteThu}. 


\section{Proofs of the main results}
Before proving Theorem~\ref{thm:main1}, we prove the following lemma.
\begin{lemma}\label{lem1}
Let $\beta$ be an algebraic number. Then
\begin{enumerate}
 \item $\Z[\beta]$ is relatively dense in $\K_\beta$;
 \item $X^\A(\beta)$ is uniformly discrete in $\K_\beta$ for any $\A\subset\Q(\beta)$ finite.
\end{enumerate}
\end{lemma}
\begin{proof}
According to Lemma 3.2. of \cite{SteThu}, $\Z[\beta]$ has a finite index in the set
$$\Oo_{\widetilde S_\beta} = \{x\in\Q(\beta):|x|_\p\leq1\text{ for all }\p\notin \widetilde S_\beta\}$$
with $\widetilde{S}_\beta = \{\p : \p | \infty \text{ or } |\beta|_\p > 1\}$. Moreover (Lemma 3.1. ibid.), $\Oo_{\widetilde S_\beta}$ is Delone in
$\widetilde\K_\beta=\prod_{\p\in\widetilde S_\beta}K_\p.$
The relatively dense set $\Oo_{\widetilde S_\beta}$ in $\widetilde\K_\beta$ is also relatively dense in $\K_\beta$, because in $\K_\beta$ it
can be perceived through the projection
\[\widetilde\K_\beta\rightarrow \K_\beta:\quad (x_\p)_{\p\in \widetilde S_\beta}\rightarrow (x_\p)_{\p\in {S}_\beta}\]
with ${S}_\beta \subseteq \widetilde S_\beta$.

For the uniform discretness of $X^\A(\beta)$ we show that the origin is not an accumulation point of $X^\A(\beta)-X^\A(\beta)=X^{\A-\A}(\beta)$ in $\K_\beta$. 
From the $\p$-adic product formula we have that
$$\prod_\p |z|_\p = \prod_{\p\in S_\beta} |z|_\p \prod_{\p\notin S_\beta} |z|_\p=1.$$
Notice that for any point of $X^{\A-\A}(\beta)$ the product over $\p\notin S_\beta$ is bounded from above by a constant dependent on $\beta$ and $\A$.  Thus the product over $\p\in S_\beta$ cannot tend to zero, implying that the origin of $\K_\beta$ is not an accumulation point of $X^{\A-\A}(\beta)$.
\end{proof}

\begin{proof}[Proof of Theorem~\ref{thm:main1}]
Consider the set
$$
\D_m=
\prod_{\substack{\p|\infty\\ |\beta|_\mathfrak{p}>1}}{B_1(0)}\times 
\prod_{\substack{\p|\infty\\ |\beta|_\p=1}}{B_m(0)}\times 
\prod_{\substack{\p\nmid \infty \\ |\beta|_\p>1}}\Oo_\p\subset\K_\beta.
$$
First we show that there is $\A\subset\Q(\beta)$, such that $\beta\D_m\subseteq\bigcup_{a\in\A}\D_m + a$  holds for any $m\geq1.$ 

Set $m=1$, then we clearly have
$$
\beta \D_1=
\prod_{\substack{\p|\infty\\ |\beta|_\mathfrak{p}>1}}{B_{|\beta|_\p}(0)}\times 
\prod_{\substack{\p|\infty\\ |\beta|_\p=1}}{B_1(0)}\times 
\prod_{\substack{\p\nmid \infty \\ |\beta|_\p>1}}\p^{\nu_\p(\beta)},
$$
where $\nu_\p(x)$ is the valuation function.
For the archimedean places it is obvious that we can find $\B_\p$, such that $B_{|\beta|_\p}(0)\subseteq \bigcup_{b\in\mathcal B_\p} B_1(0) + b$ holds. For the non-archimedean places even a stronger property holds: 
$\p^{\nu_\p(\beta)} = \bigcup_{b\in\mathcal B_\p} \Oo_\p + b.$
Then a cover of $\beta\mathcal D_1$ can be constructed through the cartesian product
\[\beta\mathcal D_1\subseteq\bigcup_{b\in\mathcal B}\mathcal D_1 + b\quad\text{ with } \mathcal B = \prod_{\p\in S_\beta} \mathcal B_\p\subset\K_\beta.\]

For $m>1$ and any archimedean place $\p$ such that $|\beta|_\p=1$ we use the following. If $B_1(0)\subseteq\bigcup_{b\in\mathcal B_\p}B_1(0)+b$ holds, then it also holds that $B_m(0)\subseteq\bigcup_{b\in\mathcal B_\p}B_m(0)+b$ (with the same $\mathcal B_\p$). This is a simple consequence of the triangle inequality. Hence the construction above yields $\beta\mathcal D_m \subseteq\bigcup_{b\in\mathcal B}\mathcal D_m+b$ with $\mathcal B$ independent of $m.$

Now we apply Proposition~\ref{prop:weak} to obtain an alphabet $\A\subset\Q(\beta)$. To make use of the approximation theorem, we need a cover with an   ``overlap'', in other words
\begin{equation}\label{eq:cover}\mathcal D_m+\delta\subseteq\bigcup_{b\in\mathcal B}\mathcal D_m+b\quad\text{ for any }|\delta|_\beta<\varepsilon.\end{equation}
Indeed, the condition \eqref{eq:cover} holds for the follwing reasons. For the archimedean places, it is apparent from the construction. For the finite places it is also trivial that 
$B_r(x) = B_r(x)+\delta$ for $\delta$ sufficiently small. Therefore the overlapping condition holds in each embedding, and consequently also in~\eqref{eq:cover}. Applying Proposition~\ref{prop:weak}, this concludes that there is $\A\subseteq\Q(\beta)$ such that 
\[\beta\D_m\subseteq\bigcup_{a\in\A}\D_m + a\quad\text{ holds for any }m\geq1.\]

Now we show how an eventually periodic $(\beta,\A)$-representation is obtained. Fix $m\geq1$. For $x\in\D_m$ we define 
\[T(x)=\beta x- D(x): \D_m\rightarrow \D_m\quad\text{ with } D(x)\in\A.\]
Defined in such way, we directly obtain that  the sequence $(|T^n(x)|_\p)_{n\in\N}$ is bounded for any $\p\in S_\beta$. Moreover, for the infinite places not belonging to $S_\beta$ we have that
$|T^n(x)|_\p=|\beta T^{n-1}(x)-a|_\p<Const(\beta,\A)$ eventually, because $|\beta|_\p<1$. For the finite places not belonging to $S_\beta$ we use the strong triangle inequality to obtain
$|T^n(x)|_\p<Const(x,\beta,\A)$ eventually. For all places not contained in $\beta,x,\A$ (these are all but finitely many places) we have $|T^n(x)|_\p=1$. Altogether $(T^n(x))_{n\in\N}$ is finite. This shows that $x\in\mathcal D_m$ admits an eventually periodic $(\beta,\A)$-representation $x=\sum_{i=1}^{+\infty}a_i\beta^{-i}$ with $a_i = D(T^{i-1}(x)).$

Given $x\in\Q(\beta),$ we can find $L\in\N$ such that $\beta^{-L}x\in \mathcal D_m$ with
\[m=\varepsilon+\max\{|x|_\p : |\beta|_\p = 1, \p\mid\infty\}.\]
Hence every $x$ has an eventually periodic $(\beta,\A)$-representation $\sum_{i=1}^{+\infty}a_i\beta^{-i+L}.$
The existence of an integer alphabet then follows from Lemma~8 of \cite{BMPV}.
\end{proof}
\begin{proof}[Proof of Theorem~\ref{thm:wg}]
Fix $\p\in S_\beta$ such that $|\beta|_\p>1$ and not corresponding to the identical embedding. For any $x$ with an eventually periodic $(\beta,\A)$-representation $x=\sum_{i=1}^{+\infty}a_i\beta^{-i}$ we have
$|x|_\p \leq \frac{\max\{|a|_\p :\ a\in\A\}}{|\beta|_\p-1}=:C(\beta,\A,\p)$ by summing the geometric series.
Since $\Q(\beta)$ is dense in $\K_\beta$, for any $c>0$, one can find $y\in\Q(\beta)$ such that $|y|_\p>C(\beta,\A,\p)$, and $|y|<c.$ Thus $y$ cannot have an eventually periodic $(\beta,\A)$-representation.

For the other direction, let $\beta$ have only one place such that $|\beta|_\p>1$.
Then following the proof of Theorem~\ref{thm:main1}, each $x\in\Q(\beta)\cap B_1(0)$ is contained in
\[\D_m = B_1(0)\times\prod_{\substack{\p|\infty\\ |\beta|_\p=1}}{B_m(0)},\quad m=\varepsilon+\max\{|x|_\p : |\beta|_\p = 1, \p\mid\infty\},\]
i.e. $x=\sum_{i=1}^{+\infty}a_i\beta^{-i}.$
%
%
\end{proof}
\begin{proof}[Proof of Theorem~\ref{thm:main2}]
The implication $1)\implies2)$ is trivial.

$2)\implies3):$ For a fixed representation $x=\sum_{i=-k}^{+\infty}a_i\beta^{-i}$ define the integral and the fractional part as

\[
\mathrm{inp}(x) = \sum_{i=-k}^{0}a_i\beta^{-i}\quad\text{ and }\quad \mathrm{frp}(x) = \sum_{i=1}^{+\infty}a_i\beta^{-i}
\]
respectively.
For any $z \in \Z[\beta]$ and any $\p\in S_\beta$ we then have an estimation for an eventually periodic representation of $z$
\[|z - \mathrm{inp}(z)|_\p = |\mathrm{frp}(z)|_\p \leq C(\beta,\A,\p).\]
Since $\mathrm{inp}(z)\in X^\A(\beta)$, we obtained that $X^\A(\beta)$ is relatively dense in $\Z[\beta]$ which is relatively dense in $\K_\beta$ by Lemma~\ref{lem1}. The statement then follows.

$3)\implies 4):$
Fix $z\in B_1(0)\subset \K_\beta$. From the relative density of $X^\A(\beta)$ we have that there are $x_n\in X^\A(\beta)$ such that 
$|\beta^n z-x_n|_\beta<r_C$ holds for all $n\in\N.$ Then $|z-\beta^{-n}x_n|_\beta<|\beta|_\beta^{-n} r_C\rightarrow 0.$
For each $n\in\N$, we obtained $(\beta,\A)$-representations of $y_n : = \beta^{-n}x_n$ as $y_n = \sum_{i=-L(n)}^{+\infty}b_i(n)\beta^{-i}$ with only finitely many nonzero digits $b_i(n).$ Moreover, 
$
z = \inp(y_n) + \frp(y_n) + \varepsilon(n)
$
with $|z|_\beta,|\frp(y_n)|_\beta,|\varepsilon(n)|_\beta$ being bounded by a constant independent of $n$. Thus $|\inp(y_n)|_\beta$ is also bounded and can take only finitely many values because of the uniform discretness of $X^\A(\beta)$. Consequently, we can write $y_n = \sum_{i=-L}^{+\infty}b_i(n)\beta^{-i}$ with $L$ being independent of chosen $z\in B_1(0)$. Clearly, the sequence $(\beta^{-L-1}y_n)_{n\in\N}\subset K(\beta,\A)$ converges to $\beta^{-L-1}z$ which belongs to $K(\beta,\A)$, because $K(\beta,\A)$ is compact. To conclude, $\beta^{-L-1}B_1(0)$ is contained in $\K_\beta.$

$4)\implies 1):$ Let $z\in\Q(\beta)$. Then for some $L\in\N$ we have that $\beta^{-L}z\in K(\beta,\A)$, i.e. $\beta^{-L}z = \sum_{i=1}^{+\infty} a_i\beta^{-i}$. Define $z_i := \beta z_{i-1} - a_i$ with $z_0 = \beta^{-L}z$ (cf. \eqref{eq:thurston}). We have that $z_n\in K(\beta,\A)$ for all $n\in\N,$ hence $(|z_n|_\p)_{n\in\N}$ is bounded for each $\p\in S_\beta.$ For the $\p\notin S_\beta$, the sequence $(|z_n|_\p)_{n\in\N}$ is bounded because $|\beta|_\p\leq1,$
and from the (strong) triangle inequality. Moreover, for almost every $\p$ we have that $|z_n|_\p=1$ for each $n\in\N$.
We conclude that $(z_n)_{n\in\N}$ takes finitely many values. The possibility of choosing an eventually periodic representation then follows from Thurston's construction \eqref{eq:thurston}.
\end{proof}
\section{Comments}
Let us conclude the paper with some comments and open questions.
\begin{enumerate}
\item Motivated by language theorerical problems, J. \v S\'ima and P. Savick\'y studied in~\cite{sisa} the so called quasi-periodic $\beta$-expansions. In a yet unpublished subsequent work they showed that the Salem root of $x^4-x^3-x^2-x+1$ allows eventually periodic $(\beta,\A)$-representations with the alphabet $\{-2,-1,0,1,2\}.$
\item The alphabet arising from the proof of Theorem~\ref{thm:main1} is ``unnecessarily'' large. Assume that $\beta$ is a complex Pisot number, i.e. $\K_\beta=\C.$ The authors of \cite{BrFrPeSv} showed in the proof of their Theorem~4.4 that $0\in\mathrm{int}(K(\beta,\A))$ for $\A=\{-M,\dots,M\}$ with $2M+1>\beta\overline\beta + |\beta+\overline\beta|$. How tight is this bound?
\item Is there a version of Theorem~\ref{thm:main2} for bases with a conjugate on the unit circle? Obviously, the last equivalence needs to be omitted in this case, for the maps $x\mapsto \beta^{-1}(x+a)$ are not contractions on $\K_\beta$ anymore.
\item Can eventually periodic representations be generated by some kind of ``simple'' dynamic system? For example, similarly to the case of the greedy expansions and shift-radix-systems acting on $\Z^d$, see~\cite{srs}.

\end{enumerate}
\section*{Acknowledgements}
The author is thankful to V. Kala and Z. Mas\'akov\'a for many useful comments. This work has been supported by Czech Science Foundation GA\v CR, grant 17-04703Y, and by Charles University Research Centre program No. UNCE/SCI/022.

\newpage

\bibliography{Biblio}
\bibliographystyle{ijmart}

\end{document}